\documentclass[JEDP]{cedram-sem}
\usepackage{amsfonts}
\usepackage{amsthm}
\usepackage{xcolor}

\newtheorem{theorem}{Theorem}[section]
\newtheorem{lemma}[theorem]{Lemma}
\newtheorem{proposition}[theorem]{Proposition}

\def \T{\mathrm{T}}
\def \P{\mathrm{P}}
\def \M{\mathrm{M}}
\def \S{\mathrm{S}}
\def \I{\mathrm{I}}
\def \bi{\dot{B}^{-1}_{2,\infty}}
\def \b{\dot{B}^{1}_{2,1}}
\def \Z{\mathbb{Z}}

\def \D{\Omega}

\def \N{\mathrm N}

\def\R{\mathbb{R}}
\def\2L{\Lambda_{\tilde{\gamma}}}
\def\1L{\Lambda_{\gamma}}

\renewcommand{\le}{\leq}
\renewcommand{\ge}{\geq}

\title 
[Unique determination of planar potentials]
{Unique determination  of the electric potential in the presence of a fixed magnetic potential in the plane}

\author
[P. \lastname Caro]
{\firstname{Pedro} \lastname{Caro}}
\address{BCAM - Basque Center for Applied Mathematics, 48009 Bilbao, Spain and Ikerbasque, Basque Foundation for Science, 48011 Bilbao, Spain} %
\email{pcaro@bcamath.org}
\author
[K. M. \lastname{Rogers}]
{\firstname{Keith} \middlename{M.} \lastname{Rogers}}
\address{Instituto de Ciencias Matem\'aticas CSIC-UAM-UC3M-UCM, 28049 Madrid, Spain}
\email{keith.rogers@icmat.es}

\thanks{Mathematics Subject Classification. Primary 35P25, 45Q05; Secondary 35J10}
\thanks{The first author was supported by MTM2015-69992-R, Severo Ochoa SEV-2017-0718 and BERC 2018-2021. The second author was supported by MTM2017-85934-C3-1-P and  Severo Ochoa SEV-2015-0554.}

\keywords{}

\begin{document}

\begin{abstract}
For potentials $V\in L^\infty(\mathbb{R}^2,\mathbb{R})$ and $A\in W^{1,\infty}(\mathbb{R}^2,\mathbb{R}^2)$ with compact support, we consider the  Schr\"odinger equation $-(\nabla +iA)^2 u+Vu=k^2u$ with fixed positive energy $k^2$.  Under a mild additional regularity hypothesis, and with fixed magnetic potential $A$, we show that the scattering solutions uniquely determine the electric potential~$V$. For this we develop the method of Bukhgeim for the purely electric Schr\"odinger equation. 
\end{abstract}

%

\maketitle

\section{Introduction}
We will assume throughout that the electric potential $V\in L^\infty(\mathbb{R}^2,\mathbb{R})$ and the magnetic potential $A\in W^{1,\infty}(\R^2,\R^2)$ have  compact support. It is a classical problem to recover~$V$ from the scattering data. Due to a gauge invariance, $A$ is not uniquely determined, however the magnetic field $\mathrm{curl} A$ could be. These problems have been studied extensively in higher dimensions; see for example \cite{DKSU, ER0, ER, H, KU, NU,NSU88, PSU, salo, salo1, salo2,sun} and the references therein. The two dimensional problem has proved more difficult and progress was made only relatively recently based on a method of Bukhgeim;  see for example \cite{A, AGTU, GT, IUY}.

Here we will not consider whether the magnetic field is uniquely determined or not. Our more modest goal will be to prove that the electric potential $V$ is uniquely determined assuming that the magnetic potential  $A$ is fixed. For the analogous two dimensional problem with $A\equiv 0$, see \cite{AFR, Blast, Blast2, Blast3, B, IN, LTV, LV, Nachman96, tej1, tej2, tej3} and the references therein.

We consider a bounded domain $\Omega\subset \mathbb{R}^2$ that contains the support of our potentials and for which $0$ is not a
Dirichlet eigenvalue for the Hamiltonian $-(\nabla +iA)^2 +V$. In this case, for all 
 $f\in H^{1/2}(\partial \Omega)$, there is a unique solution $u\in H^1(\Omega)$ to the Dirichlet problem
  \begin{equation}\label{dp}
\begin{cases}
(\nabla +iA)^2 u=Vu\\
u\big|_{\partial \Omega}=f,
\end{cases}
\end{equation}
and the Dirichlet-to-Neumann ({\small DN}) 
map $\Lambda_{V}$ can be formally defined by $$\Lambda_{V}\,:\,f\mapsto
(\nabla u \cdot n+iA\cdot n\,u)|_{\partial\Omega};$$
see the appendix for more details. Now if
 $u, v\in H^1(\D)$ satisfy $$(\nabla +iA)^2 u=V_1u\quad \text{and} \quad (\nabla +iA)^2 v=V_2v,$$ then we have an Alessandrini-type identity
\begin{equation}\label{Alessandrini}
\int_{\partial\Omega} \big(\Lambda_{V_1}-\Lambda_{V_2}\big)[u]\, \overline{v}=
\int_\D  \big(V_1-V_2\big) u \,\overline{v}.
\end{equation}
When the boundary and solutions are sufficiently regular this follows from Green's identity (and is almost direct from the rigorous definition of the {\small DN} map). 
Assuming that we can conclude that the left-hand side of this expression is zero, our problem reduces to constructing solutions $u$ and~$v$ for which the right-hand side converges to a constant multiple of $V_1-V_2$, allowing us to conclude that $V_1=V_2$.

For this we will require that both the magnetic field and the electric potential have some additional regularity which we measure in $L^2$-Sobolev spaces with norm given by $$\|f\|_{H^s}:=\|(\I-\Delta)^{s/2}f\|_{L^2},$$ where the fractional derivatives are defined $(\I-\Delta)^{s/2}f:=((1+|\cdot|^2)^{s/2}\widehat{f}\,)$ via the Fourier transform as usual.

\begin{theorem}\label{main} Suppose additionally that  $V_1,V_2,\mathrm{curl}A\in H^{s}$ for some $s>0$. Then
$$\Lambda_{V_1}=\Lambda_{V_2}\ \Rightarrow\  V_1=V_2 \quad \textrm{a.e.} \ x \in \Omega.
$$
\end{theorem}

The assumption $\Lambda_{V_1}=\Lambda_{V_2}$ ensures that the left-hand side of \eqref{Alessandrini} is zero. In the appendix we will arrive to the same conclusion by instead assuming that the outgoing scattering solutions coincide.

In what remains of the introduction we sketch the proof of Theorem~\ref{main} assuming the more technical results that will later follow. 

For the Schr\"odinger equation with purely electric potential, Bukhgeim \cite{B} considered solutions of the form $u = e^{i\psi}\big(1+w\big)$, where from now on
$$\psi(z)\equiv \psi_{\tau,x}(z)=\tfrac{\tau}{8}(z-x)^2,\qquad z\in \mathbb{C},\ \ x\in\Omega.$$
We modify his approach, instead considering solutions to 
$
(\nabla +iA)^2 u=V_1u
$
of the form $$u_1 = e^{i(\psi-\overline{\partial}^{-1}\!A)}\big(1+w_1\big).$$ Here $\overline{\partial}^{-1}$ denotes a constant multiple of the Cauchy transform which inverts $\overline{\partial}=\partial_{z_1}+i\partial_{z_2}$. 
In the following section, we prove that we can take $w\equiv w_{\tau,x}\in
H^s$ with a bound for the norm that tends to zero as~$\tau\to\infty$. This was first proven for purely electric potentials by Bl\aa sten \cite{Blast}.  

The same procedure yields solutions to 
$
(\nabla +iA)^2 u=V_2u
$
of the form $$u_2=e^{-i(\psi+\overline{\partial}^{-1}\!A)}\big(1+w_2\big).$$
Plugging these solutions, which are also in $H^1(\Omega)$, into \eqref{Alessandrini} yields
\begin{equation}\label{ale}
0=\int e^{i({\psi}+\overline{{\psi}})}e^{i(\partial ^{-1}\!\overline{A}-\overline{\partial}^{-1}\!A)}\big(V_1-V_2\big)(1+w_1)(1+\overline{w}_2)\,.
\end{equation}
Note that the integral on the right-hand side of \eqref{ale} depends on $x$ through $\psi$, $w_1$ and $w_2$.
We will see that, after multiplying this identity by a constant multiple of~$\tau$ and letting $\tau\to\infty$, the left-hand side converges in $L^2_{x}$ to $$e^{i(\partial ^{-1}\!\overline{A}-\overline{\partial}^{-1}\!A)}(V_1-V_2),$$ so that this expression must also be zero almost everywhere.
Now as $\partial ^{-1}\!\overline{A}-\overline{\partial}^{-1}\!A$ is bounded, we can conclude that $V_1-V_2=0$ almost everywhere, uniquely determining the electric potential.

\vspace{1em}

\noindent{\bf Acknowledgement:} The authors thank Mikko Salo for helpful correspondence regarding the connection with scattering and Alexey Agaltsov for bringing pertinent references to their attention.

\section{Magnetic Bukhgeim solutions}

We rewrite the Schr\"odinger equation $(\nabla +iA)^2 u=Vu$ as
$$
(\partial +i\overline{A})(\overline{\partial}+iA)u=(V-\mathrm{curl}A)u,
$$
where on the left-hand side, we have identified the vector $(A_1,A_2)$ with the complex number $A_1+iA_2$ and $\overline{A}=A_1-iA_2$. Considering solutions of the form $$u = e^{i(\psi-\overline{\partial}^{-1}\!A)}\big(1+w\big),$$ this is equivalent to 
$$
e^{-i(\psi-\overline{\partial}^{-1}\!A)}(\partial +i\overline{A})(\overline{\partial}+iA)\big[e^{i(\psi-\overline{\partial}^{-1}\!A)}w\big]=(V-\mathrm{curl}A)(1+w),
$$
using that $\overline{\partial} \psi = 0$.
Noting that 
$$
\partial +i\overline{A}=e^{-i\partial ^{-1}\!\overline{A}}\partial \big[e^{i\partial ^{-1}\!\overline{A}}\cdot\big]\quad \text{and}\quad \overline{\partial}+iA=e^{-i\overline{\partial}^{-1}\!A}\overline{\partial}\big[
e^{i\overline{\partial}^{-1}\!A}\cdot\big],
$$
we can rewrite this as
$$
e^{-i(\psi-\overline{\partial}^{-1}\!A)}e^{-i\partial ^{-1}\!\overline{A}}\partial \Big[e^{i\partial ^{-1}\!\overline{A}}e^{-i\overline{\partial}^{-1}\!A}\overline{\partial}
[e^{i\psi}w]\Big]=(V-\mathrm{curl}A)(1+w).
$$
Moreover we can write 
$$
\partial =e^{-i\overline{\psi}}\partial \big[e^{i\overline{\psi}}\cdot\big],
$$
so that this can be rewritten as
$$
e^{-i(\psi+\overline{\psi})}e^{-i(\partial ^{-1}\!\overline{A}-\overline{\partial}^{-1}\!A)}\partial \Big[e^{i(\psi+\overline{\psi})}e^{i(\partial ^{-1}\!\overline{A}-\overline{\partial}^{-1}\!A)}\overline{\partial}w\Big]=(V-\mathrm{curl}A)(1+w).
$$
To solve this in $\Omega$, we define the inverse conjugated Laplacian $\Delta^{-1}_{\psi}$ by 
$$
\Delta^{-1}_{\psi} F:=\overline{\partial}^{-1}\left[\mathbf{1}_Qe^{-i(\psi+\overline{\psi})}e^{-i(\partial ^{-1}\!\overline{A}-\overline{\partial}^{-1}\!A)}\partial^{-1} \big[\mathbf{1}_Qe^{i(\psi+\overline{\psi})}e^{i(\partial ^{-1}\!\overline{A}-\overline{\partial}^{-1}\!A)}F\big]\right],
$$
where $Q$ is an axis parallel square that contains $\Omega$, and look for $w$ that satisfy 
$$
w=\Delta_{\psi}^{-1}[(V-\mathrm{curl}A)(1+w)].
$$
Defining $\S_\tau[F]=\Delta^{-1}_{\psi}[(V-\mathrm{curl}A)F]$, this can be rewritten as
$$
(\I-\S_\tau)w=\Delta^{-1}_{\psi}[V-\mathrm{curl}A]
$$
and if we can show that $\S_\tau$ is a contraction we can invert $(\I-\S_\tau)$ via Neumann series, yielding
$$
w=(\I-\S_\tau)^{-1}\Delta^{-1}_{\psi}[V-\mathrm{curl}A].
$$

We look for this contraction, with $\tau$ sufficiently large, in the homogeneous Sobolev space~$\dot{H}^s$ with  norm $\|f\|_{\dot{H}^s}=\|(-\Delta)^{s/2}f\|_{2}$. First write
$$\Delta^{-1}_{\psi}=
\overline{\partial}^{-1}\circ \M_{-{\tau}}\circ \N_{-A}\circ\partial ^{-1} \circ
\M_{{\tau}}\circ \N_A$$ where the multiplier operators $\M_{\pm{\tau}}$ and $\N_{\pm{A}}$ are defined by
$$
\M_{\pm{\tau}}[F]=
 \mathbf{1}_Qe^{\pm i({\psi}+\overline{{\psi}})} F\quad \text{and}\quad \N_{\pm{A}}[F]=
 e^{\pm i(\partial ^{-1}\!\overline{A}-\overline{\partial}^{-1}\!A)}F.
$$
We will need good estimates for these multiplier operators. 
\begin{lemma}\label{isitgood}\cite{AFR}
 Let $0< s_1,s_2<1$. Then
\begin{equation*}\|\M_{\pm \tau}[F]\|_{\dot{H}^{-s_2}}\le C\tau^{-\min\{s_1,s_2\}}\|F\|_{\dot{H}^{s_1}},\quad \ \tau\ge 1.
\end{equation*}
\end{lemma}

\begin{proof} 
By the H\"older and Hardy--Littlewood--Sobolev inequalities (see for example \cite[pp. 354]{van}), we have
\begin{equation}\label{ds}
\|\M_{\pm \tau}[F]\|_{2}\le C\|F\|_{\dot{H}^{s_1}},
\end{equation}
and 
\begin{equation}\label{fs}
\|\M_{\pm \tau}[F]\|_{\dot{H}^{-s_2}}\le C\|F\|_{2},
\end{equation}
with $0\le s_1,s_2<1$.
So by complex interpolation, it will suffice to prove that
\begin{equation}\label{rs}\|\M_{\pm \tau}[F]\|_{\dot{H}^{-s}}\le C\tau^{-s}\|F\|_{\dot{H}^{s}}.
\end{equation}
Indeed, if $s_2<s_1$ we interpolate with \eqref{ds}, taking $s=s_1$, and if $s_1<s_2$ we interpolate with \eqref{fs}, taking $s=s_2$.
Now by real interpolation with the trivial $L^2$ bound, \eqref{rs} would follow from
\begin{equation}\label{besov}
\|\M_{\pm \tau}F\|_{\dot{B}^{-1}_{2,\infty}}\le C\tau^{-1}\,\|F\|_{\dot{B}^{1}_{2,1}}
\end{equation}
(see Theorem 6.4.5 in \cite{BL}), where the Besov norms are defined as usual by
$$
\|f\|_{\bi}=\sup_{j\in\Z} 2^{-j}\|\P_jf\|_{L^2}\quad\quad \text{and}\quad\quad \|f\|_{\b}=\sum_{j\in\Z} 2^{j}\|\P_jf\|_{L^2}.
$$
Here, $\widehat{\P_j f}= \vartheta(2^{-j} |\cdot| ) \widehat{f}$  with $\vartheta$ satisfying $\text{supp}\, \vartheta \subset (1/2,2)$ and
$$
\sum_{j\in\Z} \vartheta(2^{-j} \cdot )=1.
$$
As
$
\|F\|_{\dot{B}^{-1}_{2,\infty}}\le C\|\widehat{F}\|_\infty
$
and
$
\|\widehat{F}\,\|_1\le C\|F\|_{\dot{B}^1_{2,1}},
$
the estimate \eqref{besov} would in turn follow from
\begin{equation}\label{pop}
\|\widehat{\M_{\pm \tau}F}\|_{\infty}\le C\tau^{-1}\,\|\widehat{F}\|_{1}.
\end{equation}
Now, by the Fourier inversion formula and Fubini's theorem,
\begin{align*}
|\widehat{\M_{\pm \tau}F}(\xi)|&=\frac{1}{(2\pi)^2}\Big|\int_Q e^{\pm i(\psi(z)+\overline{\psi(z}))}\int \widehat{F}(\omega)\,e^{iz\cdot \omega} d\omega\, e^{-iz\cdot \xi} dz\Big|\\
&\le \int\Big|\int_Q
e^{\pm i\tau\frac{(z_1-x_1)^2-(z_2-x_2)^2}{4}}\,  e^{iz\cdot (\omega-\xi)}
dz\Big| |\widehat{F}(\omega)|\, d\omega
\end{align*}
so that \eqref{pop} follows by two applications of van der Corput's lemma \cite[pp. 332]{van} (factoring the integral into the product of two integrals).
\end{proof}

We will also need the following lemma which is a consequence of properties of the Cauchy transform.

\begin{lemma}\label{easy} Let $\mathrm{curl} A\in L^2(\Omega)$ and $0\le s\le 1$. Then
\begin{equation*}\|\N_{\pm A}[F]\|_{\dot{H}^{s}}\le C_{\!A}\|F\|_{\dot{H}^{s}}.
\end{equation*}
\end{lemma}

\begin{proof}
It is easy to calculate that
\begin{equation*}
-\Delta(\partial ^{-1}\!\overline{A}-\overline{\partial}^{-1}\!A)=-\overline{\partial}\,\overline{A}+\partial A=2i\mathrm{curl} A.
\end{equation*}
Thus, by Sobolev embedding, we have that
\begin{align*}
\|\nabla(\partial ^{-1}\!\overline{A}-\overline{\partial}^{-1}\!A)\|_{4}
&\le C\|\mathrm{curl} A\|_{4/3}\le C\|\mathrm{curl} A\|_{2}<\infty.
\end{align*}
and 
\begin{align*}
\|\partial ^{-1}\!\overline{A}-\overline{\partial}^{-1}\!A\|_{\infty}\le C\|\mathrm{curl} A\|_{2}<\infty.
\end{align*}
Now $$\nabla(\N_{\pm A}[F])=\pm i\nabla(\partial ^{-1}\!\overline{A}-\overline{\partial}^{-1}\!A)e^{\pm i(\partial ^{-1}\!\overline{A}-\overline{\partial}^{-1}\!A)}F+e^{\pm i(\partial ^{-1}\!\overline{A}-\overline{\partial}^{-1}\!A)}\nabla F$$ so that, by H\"older's inequality and Sobolev embedding,
\begin{align*}
\|\N_{\pm A}[F]\|_{\dot{H}^1}&\le C\|\nabla(\partial ^{-1}\!\overline{A}-\overline{\partial}^{-1}\!A)\|_4\| F\|_{L^4(\Omega)} + \|F\|_{\dot{H}^1}\\
&\le C_{\!A}\|\nabla F\|_{L^{4/3}(\Omega)} + \|F\|_{\dot{H}^1}\le C_{\!A}\|F\|_{\dot{H}^1}.
\end{align*}
On the other hand, we obviously have
$$
\|\N_{\pm A}[F]\|_{2}\le C_{\!A}\|F\|_{2}.
$$
Interpolating between the two estimates for $\N_{\pm A}$ gives the result.
\end{proof}

In the following lemma, we apply both of the previous lemmas twice enabling us to maximise the decay in $\tau$. This is necessary in order to show later that the terms involving~$w$ arising in Alessandrini's identity \eqref{ale} are indeed decay terms.

\begin{lemma}\label{celtic7} Let $0<s<1$. Then
$$
\|\Delta_\psi^{-1}[F]\|_{\dot{H}^{s}}\le C\tau^{-1}\|F\|_{\dot{H}^{s}},\quad \ \tau\ge 1.
$$
\end{lemma}

\begin{proof} By two applications of each of the previous two lemmas,
\begin{align*}
\|\Delta_\psi^{-1}\|_{\dot{H}^{s}\to \dot{H}^{s}}
&\le \|\M_{-{\tau}}\circ \N_{-A}\circ\partial ^{-1} \circ
\M_{{\tau}}\circ \N_A\|_{\dot{H}^{s}\to \dot{H}^{s-1}}\\
&\le C\tau^{s-1}\|\N_{-A}\circ\partial ^{-1} \circ
\M_{{\tau}}\circ \N_A\|_{\dot{H}^{s}\to \dot{H}^{1-s}}\\
&\le C\tau^{s-1}\|
\M_{{\tau}}\circ \N_A\|_{\dot{H}^{s}\to \dot{H}^{-s}}\\
&\le C\tau^{s-1-s}\| \N_A\|_{\dot{H}^{s}\to \dot{H}^{s}}=C\tau^{-1},
\end{align*}
and we are done.
\end{proof}

In the following lemma, we use Lemma~\ref{isitgood} only once, and
gain some integrability using the Hardy--Littlewood--Sobolev
theorem; see for example \cite[pp. 354]{van}). By taking~$\tau$ sufficiently large, we obtain our contraction
and thus our magnetic Bukhgeim solution  as described above.

\begin{lemma}\label{celtic3} Let $0<s<1$. Then
$$
\|\S_\tau[F]\|_{\dot{H}^s}\le C\tau^{-\min\{2s,1-s\}}\|V-\mathrm{curl} A\|_{\dot{H}^{s}}\|F\|_{\dot{H}^s},
$$
whenever $\tau\ge1$ and 
$$
\|\S_\tau[F]\|_{\dot{H}^\frac{1-s}{2}}\le C\tau^{-\frac{1+s}{2}}\|V-\mathrm{curl} A\|_{\dot{H}^{s}}\|F\|_{\dot{H}^\frac{1-s}{2}}.
$$
\end{lemma}

\begin{proof} By the Cauchy--Schwarz and Hardy--Littlewood--Sobolev inequalities,
$$
\|(V-\mathrm{curl} A)F\|_p\le \|V-\mathrm{curl} A\|_{2p}\|F\|_{2p}\le C\|V-\mathrm{curl} A\|_{\dot{H}^s}\|F\|_{\dot{H}^s},
$$
where $p=\frac{1}{1-s}$. Thus, as $\S_\tau[F]=\Delta_\psi^{-1}[(V-\mathrm{curl} A)F]$, for the first inequality it will suffice to prove that
$$
\|\Delta_\psi^{-1}\|_{L^p\to \dot{H}^{s}}\le C\tau^{{-\min\{2s,1-s\}}}.
$$
When $0<s<1/3$, by Lemma~\ref{isitgood}, we have
\begin{align*}
\|\Delta_\psi^{-1}\|_{L^p\to \dot{H}^{s}}
&\le \|\M_{-{\tau}}\circ \N_{-A}\circ\partial ^{-1} \circ
\M_{{\tau}}\circ \N_A\|_{L^p\to \dot{H}^{s-1}}\\
&\le C\tau^{-2s}\|N_{-A}\circ\partial ^{-1} \circ
\M_{{\tau}}\circ \N_A\|_{L^p\to \dot{H}^{2s}}\\
&\le C\tau^{-2s}\|\M_{{\tau}}\circ \N_A\|_{L^p\to \dot{H}^{2s-1}}\\
&\le C\tau^{-2s}\|\M_{{\tau}}\circ \N_A\|_{L^p\to L^p}, 
\end{align*}
where the final inequality is by Hardy--Littlewood--Sobolev.
When $s\ge 1/3$, we also use H\"older's inequality at the end;
\begin{align*}
\|\Delta_\psi^{-1}\|_{L^p\to \dot{H}^{s}}
&\le \|\M_{-{\tau}}\circ \N_{-A}\circ\partial ^{-1} \circ
\M_{{\tau}}\circ \N_A\|_{L^p\to \dot{H}^{s-1}}\\
&\le C\tau^{s-1}\|N_{-A}\circ\partial ^{-1} \circ
\M_{{\tau}}\circ \N_A\|_{L^p\to \dot{H}^{1-s}}\\
&\le C\tau^{s-1}\|\M_{{\tau}}\circ \N_A\|_{L^{p}\to \dot{H}^{-s}}\\
&\le C\tau^{s-1}\|\M_{{\tau}}\circ \N_A\|_{L^p\to L^{p^*}},
\end{align*}
where $p^*=\frac{2}{s+1}$, and so we are done.

For the second inequality we note that by the H\"older and the Hardy--Littlewood--Sobolev inequalities,
$$
\|(V-\mathrm{curl} A)F\|_q\le C\|V-\mathrm{curl} A\|_{\dot{H}^s}\|F\|_{\dot{H}^\frac{1-s}{2}},
$$
where $q=\frac{4}{3-s}$. Thus, as $\S_\tau[F]=\Delta_\psi^{-1}[(V-\mathrm{curl} A)F]$, for the second inequality it will suffice to prove that
$$
\|\Delta_\psi^{-1}\|_{L^q\to \dot{H}^{\frac{1-s}{2}}}\le C\tau^{-\frac{1+s}{2}}.
$$
Again, by Lemma~\ref{isitgood}, we have
\begin{align*}
\|\Delta_\psi^{-1}\|_{L^q\to \dot{H}^{\frac{1-s}{2}}}
&\le \|\M_{-{\tau}}\circ \N_{-A}\circ\partial ^{-1} \circ
\M_{{\tau}}\circ \N_A\|_{L^q\to \dot{H}^{-\frac{1+s}{2}}}\\
&\le C\tau^{-\frac{1+s}{2}}\|N_{-A}\circ\partial ^{-1} \circ
\M_{{\tau}}\circ \N_A\|_{L^q\to \dot{H}^{\frac{1+s}{2}}}\\
&\le C\tau^{-\frac{1+s}{2}}\|\M_{{\tau}}\circ \N_A\|_{L^q\to \dot{H}^{\frac{s-1}{2}}}\\
&\le C\tau^{-\frac{1+s}{2}}\|\M_{{\tau}}\circ \N_A\|_{L^q\to L^{q}},
\end{align*}
where the final inequality is by Hardy--Littlewood--Sobolev, and so the proof is complete.
 \end{proof}

It remains to show
that a constant multiple of the right-hand side of Alessandrini's identity \eqref{ale} converges;  $$\T^\tau_{(1+w_1)(1+\overline{w}_2)} [V_1-V_2] \to e^{i(\partial ^{-1}\!\overline{A}-\overline{\partial}^{-1}\!A)}(V_1-V_2)$$ as $\tau$ tends to infinity, where
$$
\T^\tau_{a}[F](x)=\frac{\tau}{4\pi}\int_{\R^2}
\M_{{\tau}}\circ\N_{A} [F](z){a(z)}\,dz.
$$
First we show that $\T_w^\tau F$ can be considered to be a remainder term.

\begin{proposition}\label{Remainder}
Let $F  \in \dot{H}^{s}$ with $0<s<1$. Then
\[ \lim_{\tau \to \infty}   \T^\tau_w[F](x)=0,\quad x\in\D.\]
Moreover, if  $\tau$ is sufficiently large, then
$$
\sup_{x\in\D}|\T^\tau_w[F](x)|\le C\tau^{-s}\|V-\mathrm{curl}A\|_{\dot{H}^{s}}\|F\|_{\dot{H}^{s}}.
$$
\end{proposition}

\begin{proof}
By Lemmas~\ref{isitgood} and \ref{easy},
\[ \begin{aligned} |\T^\tau_w[F](x)| & \le  C\tau  \| \M_{{\tau}}\circ\N_{A} [F]\|_{\dot{H}^{-s}}\|w\|_{\dot{H}^{s}}
\\  & \le  C\tau^{1-s}  \|F\|_{\dot{H}^{s}}  \|(\I-\S_\tau)^{-1}\S_\tau[1]\|_{\dot{H}^{s}}.
\end{aligned}\]
By Lemma \ref{celtic3},  we can treat
 $(\I-\S_\tau)^{-1}$ by Neumann series to deduce  that it is a bounded operator on $\dot{H}^{s}$ whenever $\tau\ge 1$ and  $$C\tau^{-\min\{2s,1-s\}}\|V-\mathrm{curl}A\|_{\dot{H}^s}\le \frac{1}{2}.$$ Then
\[\begin{aligned}
|\T^\tau_w[F](x)| & \le  C\tau^{1-s}  \|F\|_{\dot{H}^{s}}  \|\Delta_\psi^{-1}[V-\mathrm{curl}A]\|_{\dot{H}^{s}}  \\
& \le   C\tau^{-s}   \|F\|_{\dot{H}^{s}}  \|V-\mathrm{curl}A\|_{\dot{H}^{s}},
\end{aligned} \]
by an application of Lemma~\ref{celtic7}, which is the desired estimate.
\end{proof}

We also need the following lemma so that the final term can also be treated as a remainder term. 

\begin{proposition}\label{Remainder2}
Let $F  \in \dot{H}^{s}$ with $0<s<1$. Then
\[ \lim_{\tau \to \infty}   \T^\tau_{w_1\overline{w}_2}[F](x)=0,\quad x\in\D.\]
Moreover, if $\tau$ is sufficiently large, then
$$
\sup_{x\in\D}|\T^\tau_{w_1\overline{w}_2}[F](x)|\le C\tau^{-s}\|V_1-\mathrm{curl}A\|_{\dot{H}^{s}}\|V_2-\mathrm{curl}A\|_{\dot{H}^{s}}\|F\|_{\dot{H}^{s}}.
$$
\end{proposition}

\begin{proof}
By H\"older's inequality, the Hardy--Littlewood--Sobolev inequality and Lemma \ref{easy}, with $q=\frac{2}{1-s}$ and $r=\frac{4}{1+s}$,
\[ \begin{aligned} |\T^\tau_{w_1\overline{w}_2}[F](x)| & \le  C\tau  \|\N_{A} [F]\|_{q}\|w_1\|_{r}\|w_2\|_{r}
\\
& \le  C\tau  \|F\|_{\dot{H}^{s}}\|w_1\|_{\dot{H}^{\frac{1-s}{2}}}\|w_2\|_{\dot{H}^{\frac{1-s}{2}}}
\\  & \le  C\tau \|F\|_{\dot{H}^{s}}  \|(\I-\S^{V_1}_\tau)^{-1}\S^{V_1}_\tau[1]\|_{\dot{H}^{\frac{1-s}{2}}}\|(\I-\S^{V_2}_\tau)^{-1}\S^{V_2}_\tau[1]\|_{\dot{H}^{\frac{1-s}{2}}}.
\end{aligned}\]
By the second inequality of Lemma \ref{celtic3},  the inverse operators can be treated by Neumann series to deduce  that they are bounded  on $\dot{H}^{\frac{1-s}{2}}$ whenever $\tau\ge 1$ and  $$C\tau^{-\frac{1+s}{2}}\|V_1-\mathrm{curl}A\|_{\dot{H}^s}+C\tau^{-\frac{1+s}{2}}\|V_2-\mathrm{curl}A\|_{\dot{H}^s}\le \frac{1}{2}.$$ 
Thus, by two applications of Lemma~\ref{celtic3},
\[\begin{aligned}
|\T^\tau_{w_1\overline{w}_2}[F](x)| & \le  C\tau \|F\|_{\dot{H}^{s}}  \|\S^{V_1}_\tau[1]\|_{\dot{H}^{\frac{1-s}{2}}}\|\S^{V_2}_\tau[1]\|_{\dot{H}^{\frac{1-s}{2}}}  \\
& \le   C\tau^{-s}  \|F\|_{\dot{H}^{s}}\|V_1-\mathrm{curl}A\|_{\dot{H}^{s}}\|V_2-\mathrm{curl}A\|_{\dot{H}^{s}},
\end{aligned} \]
 which is the desired estimate.
\end{proof}

 Noting that $e^{i({\psi(z)}+\overline{{\psi(z}}))}=
\exp\big(i\tau\frac{(z_1-x_1)^2-(z_2-x_2)^2}{4}\big)$, it remains to prove
\begin{equation}\label{ipo}
\lim_{\tau\to \infty} \T^\tau_{1} [F]=\N_{A} [F]
\end{equation}
in $L^2(\mathbb{R}^2)$, 
where $\T_1^{\tau}$ is defined by
\begin{equation*}\label{no}
\T_1^{\tau}[F](x)=\frac{\tau}{4\pi}\int
\exp\big(i\tau\tfrac{(z_1-x_1)^2-(z_2-x_2)^2}{4}\big) \N_{A}[F](z)\,dz.
\end{equation*}
Now when $F$ is a Schwartz function, this expression is equal to
$e^{i\frac{1}{\tau}\Box}\N_{A}[F](x)$, where
\begin{equation*}\label{nonon}
e^{i\frac{1}{\tau}\Box}[G](x)=\frac{1}{(2\pi)^2}\int_{\R^2}e^{ix\cdot
\xi}\, e^{-i\frac{1}{\tau}(\xi_1^2-\xi_2^2)} \,\widehat{G}(\xi)\, d\xi.
\end{equation*}
We see that $\T^\tau_{1}[F]$ solves the time-dependent
nonelliptic Schr\"odinger equation, $$i\partial_t u +\Box u=0,$$ where
$\Box =\partial_{x_1x_1}-\partial_{x_2x_2}$, with initial data $\N_{A}[F]$
at time $1/\tau$. Thus the desired convergence is a well--known consequence of the time-dependent Schr\"odinger theory.

\appendix

\section{The {\small DN} map and the scattering data}\label{samp}

If zero is not a Dirichlet eigenvalue,  then there is a unique weak solution to the Dirichlet problem~\eqref{dp} that satisfies
 \begin{equation}\label{ref}
 \|u\|_{H^1(\Omega)}\le C\|f\|_{H^{1/2}(\partial \Omega)}.
 \end{equation}
Here $H^{1/2}(\partial \Omega):=H^1(\Omega)/H^1_0(\Omega)$, where $H^1_0(\Omega)$ denotes the closure of
$C^\infty_0(\Omega)$ in $H^1(\Omega)$.
The {\small DN} map $\Lambda_{V}$, taking values in the dual of $H^{1/2}(\partial \Omega)$, is then defined by
\begin{equation}\label{what} \Big \langle \Lambda_{V}[f],\psi \Big \rangle_{\!\partial\Omega} =\int_{\Omega} (A^2+V) u\overline{\Psi} +\nabla u\cdot \nabla \overline{\Psi} +iA\cdot(u\nabla \overline{\Psi}-\overline{\Psi}\nabla u)\end{equation}
for all $\Psi\in H^1(\Omega)$ with $\psi=\Psi+H^1_0(\Omega)$.
When the solution and boundary are sufficiently smooth, this
definition coincides with that of the introduction by Green's
identity. 

The relevant scattering question considers the Schr\"odinger equation at a fixed positive energy~$k^2$. That is to say, we  consider the equation
\begin{equation}\label{scod}
-(\nabla +iA)^2 u+Vu=k^2u.
\end{equation}
The outgoing scattering solutions $\Psi_{\!\theta}$ are perturbations of the plane waves $e^{ikx\cdot\theta}$; we refer to \cite{PSU} for the precise definition.  Supposing that they are the same for two different electric potentials $V_1$ and $V_2$,  we have in particular that
$$
V_1\Psi_{\!\theta}= (\nabla +iA)^2 \Psi_{\!\theta}+k^2\Psi_{\!\theta}=V_2\Psi_{\!\theta},
$$
so that
\begin{equation}\label{al}
\int_\Omega \big(V_1-V_2\big)\Psi_{\!\theta}\overline{v}=0.
\end{equation}
On the other hand, any solution $u\in H^1(\Omega)$ of \eqref{scod} can be approximated in $L^2(\Omega)$ by the scattering solutions (see the proof of \cite[Proposition 2.4]{PSU}), so we can approximate
\begin{equation}\label{int}
\int_{\Omega}\big(V_1-V_2\big)u\overline{v}
\end{equation}
 by a sequence that takes the form of the left-hand side of \eqref{al}. Hence we deduce that, if the scattering solutions coincide, then the integrals appearing in \eqref{int} and \eqref{Alessandrini} must be identically zero.
 
Note that we have completely bypassed the associated {\small DN} maps $\Lambda_{V_1-k^2}$ and $\Lambda_{V_2-k^2}$ in order to conclude that \eqref{int} is zero. This was made possible by assuming that the scattering solutions are equal everywhere. Alternatively we could have supposed the weaker hypothesis 
that the scattering solutions {\it coincide at infinity}. That is to say that the {\it scattering amplitudes} are equal. The usual argument then consists of deducing from this hypothesis that the {\small DN} maps also coincide; see \cite{AFR1} for such an argument in the purely electric case. Certain resolvent estimates are required for this argument, and it appears that less is known regarding estimates of this type for the magnetic Schr\"odinger equation.

\end{document}